\documentclass[10pt,reqno]{amsart}

\usepackage{amsmath,amscd,amssymb }
\usepackage{tikz-cd}

\usepackage[backend=bibtex,style=numeric]{biblatex}
\usepackage{enumerate}
\bibliography{shortcovers.bib}

\newtheorem{proposition}{Proposition}[section]
\newtheorem{lemma}[proposition]{Lemma}
\newtheorem{theorem}[proposition]{Theorem}
\newtheorem{corollary}[proposition]{Corollary}

\theoremstyle{definition}
\newtheorem{remark}[proposition]{Remark}
\newtheorem{definition}[proposition]{Definition}
\newtheorem{example}[proposition]{Example}

\newcommand{\C}{\mathbb{C}}

\newcommand{\Z}{\mathbb{Z}}
\newcommand{\Q}{\mathbb{Q}}

\newcommand{\pr}{\mathbb{P}}

\newcommand{\scL}{\mathcal{L}}

\newcommand{\scD}{\mathcal{D}}
\newcommand{\scH}{\mathcal{H}}

\newcommand{\scX}{\mathcal{X}}

\newcommand{\scY}{\mathcal{Y}}

\newcommand{\scW}{\mathcal{W}}

\DeclareMathOperator{\Aut}{Aut}

\DeclareMathOperator{\DF}{DF}

\DeclareMathOperator{\Spec}{Spec}

\title[On K-stability of finite covers]{On K-stability of finite covers}

\author[Ruadha\'i Dervan]{Ruadha\'i Dervan}

\begin{document}

\maketitle


\begin{abstract} We show that certain Galois covers of K-semistable Fano varieties are K-stable. We use this to give some new examples of Fano manifolds admitting K\"ahler-Einstein metrics, including hypersurfaces, double solids and threefolds.
\end{abstract}

\section{Introduction}

One of the foremost results in complex geometry in recent years is the proof of the Yau-Tian-Donaldson conjecture on Fano manifolds, which states that a Fano manifold admits a K\"ahler-Einstein metric if and only if it is K-polystable, which is a completely algebro-geometric notion \cite{RB,CDS,GT}. While this result is of great theoretical interest, there is currently only one known way of directly proving a Fano variety is K-stable \cite{OS}. The aim of the present work is to give another way, by relating K-stability of a Fano variety to K-stability of its finite covers. 

\begin{theorem}\label{intromaintheorem} Let $Y\to X$ be a cyclic Galois covering of smooth Fano varieties with smooth branch divisor $D\in |-\lambda K_X|$ for $\lambda \geq 1$. If $X$ is K-semistable, then $Y$ is K-stable. \end{theorem}

It is interesting to note that this implies $Y$ has finite automorphism group, which we show in Corollary \ref{automcor}. We remark that it is not known in general that K-stable varieties have finite automorphism group. This would follow in general provided one knew that that the automorphism group of a K-stable variety is reductive, however this is an open problem. The reductivity required follows in our case from the Yau-Tian-Donaldson conjecture.

Also through the Yau-Tian-Donaldson conjecture, we obtain the following analytic result, which was first proven by Li-Sun \cite[Theorem 6.1]{LS}.

\begin{corollary}\label{KEmetrics}  Let $Y\to X$ be a cyclic Galois covering of Fano manifolds with smooth branch divisor $D\in |-\lambda K_X|$ for $\lambda \geq 1$. If $X$ admits a K\"ahler-Einstein metric, then so does $Y$. \end{corollary}


Perhaps the simplest example of a Fano variety satisfying the hypotheses of Theorem \ref{intromaintheorem} is the Fano hypersurface \[Y=V(x_0^{n}+f(x_1,\hdots, x_n))\subset \pr^{n}\] with $f(x_1,\hdots, x_{n+1})$ a homogeneous polynomial of degree $n$.  $Y$ admits a degree $n$ map to $\pr^{n-1}$ by forgetting the first coordinate, and in this case the Galois group is just $\Z/n\Z$ acting by multiplication of $x_0$ by an $n^\textit{th}$ root of unity. Projective space is K-semistable (in fact K-polystable) as it admits a K\"ahler-Einstein metric. The branch locus is the anti-canonical divisor $V(f(x_1,\hdots, x_{n})) \subset \pr^{n-1}$, so provided this is smooth, Theorem \ref{intromaintheorem} implies $Y$ is K-stable and hence admits a K\"ahler-Einstein metric. Special cases of such hypersurfaces were already known to admit K\"ahler-Einstein metrics \cite{AN,GT2,TY,CPW,AGP,AP}, however in this generality the fact that $Y$ admits a K\"ahler-Einstein metric is new. We give some more new examples in Section \ref{examples}. For general polarised varieties $(X,L)$, i.e. with $L$ an ample line bundle, Wang-Zhou \cite{WZ} have given examples of K-polystable toric surfaces which by Donaldson's work \cite{SD2} admit constant scalar curvature K\"ahler metrics. Again for general $(X,L)$, the first non-toric examples of K-stable varieties which are not yet known to admit constant scalar curvature K\"ahler metrics were given by the author \cite{RD2}. 

A similar result to Corollary \ref{KEmetrics} regarding the existence of K\"ahler-Einstein metrics on finite covers has been proven by Arezzo-Ghigi-Pirola \cite{AGP}. Their result requires that $Y$ admits quotient maps to several Fano manifolds, however the conditions on the branch divisors are weaker than we require. 

The proof of Theorem \ref{intromaintheorem} uses the theory of log K-stability of pairs $(X,D)$ for $D$ a divisor, which proved to be a crucial technical tool in the proof of the Yau-Tian-Donaldson conjecture. We first prove the following result, which is an algebro-geometric analogue of \cite[Theorem 6.1]{LS}. 

\begin{theorem}\label{introequivalence} Let $Y\to X$ be a cyclic Galois cover of Fano varieties of degree $m$ with $Y$ smooth and branch divisor $D\subset X$. If $(X,D)$ is log K-stable with angle $\frac{1}{m}$, then $Y$ is K-stable. \end{theorem}

Having proven Theorem \ref{introequivalence}, a simple monotonicity argument then gives Theorem \ref{intromaintheorem}. 

\noindent {\bf Notation and conventions:} We work throughout over the complex numbers and use additive notation for line bundles. 

\
 
\noindent {\bf Acknowledgements:}  I would like to thank my supervisor, Julius Ross, for several useful discussions and for suggesting the topic. I would also like to thank Giulio Codogni and the anonymous referee for helpful comments. Funded by a studentship associated to an EPSRC Career Acceleration Fellowship (EP/J002062/1) and a Fondation Wiener-Anspach scholarship.

\section{Preliminaries on K-stability}

Let $(X,L)$ be a normal polarised variety, i.e. $L$ is an ample line bundle on $X$. We say $(X,L)$ is K-stable, if loosely speaking a certain weight (called the Donaldson-Futaki invariant) associated to each flat degeneration (called a test configuration) is positive. As several of our arguments in proving Theorem \ref{intromaintheorem} apply to general polarised varieties, we discuss K-stability in this generality rather than focussing on the Fano case.

\begin{definition} A \emph{test configuration} $(\scX,\scL)$ for a normal polarised variety $(X,L)$ is a normal variety $\scX$ together with
\begin{itemize} 
\item a proper flat morphism $\pi: \scX \to \C$,
\item a $\C^*$-action on $\scX$ covering the natural action on $\C$,
\item and an equivariant relatively ample line bundle $\scL$ on $\scX$
\end{itemize}
such that the fibre $(\scX_t,\scL_t)$ over $t\in\C$ is isomorphic to $(X,L^r)$ for one, and hence all, $t \in \C^*$ and for some $r>0$. We call $r$ the \emph{exponent} of the test configuration.
\end{definition}

\begin{remark} One can canonically compactify a test configuration to a family over $\pr^1$ as follows. Over $\C^*$, a test configuration is isomorphic to a trivial family $\scX\backslash\scX_0\cong X\times\C^*$. Gluing the trivial family around infinity therefore gives a well defined family over $\pr^1$.  \end{remark}

As the $\C^*$-action on $(\scX,\scL)$ fixes  the central fibre $(\scX_0,\scL_0)$, there is an induced action on $H^0(\scX_0,\scL^k_0)$ for all $k$. Denote by $w(k)$ the total weight of this action, which is a polynomial in $k$ of degree $n+1$ for $k\gg 0$, where $n$ is the dimension of $X$. Denote the Hilbert polynomial of $(\scX_0,\scL_0)$ as \[h(k)=a_0k^n+a_1k^{n-1}+O(k^{n-2}),\] and denote also the total weight of the $\C^*$-action on $H^0(\scX_0,\scL^k_0)$ as \[w(k) = b_0k^{n+1}+b_1k^n+O(k^{n-1}).\]

\begin{definition}We define the \emph{Donaldson-Futaki invariant} of a test configuration $(\scX,\scL)$ to be \[\DF(\scX,\scL)=\frac{b_0a_1-b_1a_0}{a_0}.\] We say $(X,L)$ is \emph{K-stable} if $\DF(\scX,\scL)>0$ for all test configurations which are not trivial, i.e. isomorphic to $X\times\C$ with the trivial action. Similarly we say $(X,L)$ is \emph{K-semistable} if for all test configurations, the Donaldson-Futaki invariant is non-negative. When $X$ is Fano, we take $L=-K_X$ and often just say $X$ is K-stable. \end{definition}

\begin{remark}\label{remarks} We refer to \cite{RT} for a more thorough introduction to K-stability and its relation to Geometric Invariant Theory.
\begin{itemize}
\item K-stability of variety is closely related to asymptotic notions of stability for varieties, and hence to Geometric Invariant Theory. The test configurations appearing in the definition of K-stability are geometrisations of the one parameter subgroups appearing in the Hilbert-Mumford criterion in Geometric Invariant Theory, with the Donaldson-Futaki invariant playing the role of the weight of a one parameter subgroup. 
\item It is known that the existence of a constant scalar curvature K\"ahler metric implies K-stability provided the automorphism group of the polarised variety is finite \cite{JS,SD}, and that the two are equivalent for Fano varieties $X$ with $L=-K_X$ (in this case the metric is K\"ahler-Einstein) \cite{RB,CDS,GT}. The Yau-Tian-Donaldson conjecture states that the existence of a constant scalar curvature K\"ahler metric is equivalent to K-polystability. However, it is now expected that one may need to strengthen the definition of K-polystability for this to be true \cite{GS} \cite[Conjecture 1.3]{RD}.
\item The definition of the Donaldson-Futaki invariant makes sense for test configurations with non-normal total spaces. Our assumption that the test configuration is normal is justified by \cite[Remark 5.2]{RT}, which states that normalisation \emph{lowers} the Donaldson-Futaki invariant. On the other hand, each polarised variety admits many non-normal test configurations with Donaldson-Futaki invariant zero \cite[Section 8.2]{LX}. These test configurations are equivalently characterised as those with norm zero \cite[Theorem 1.2]{RD} \cite{BHJ}.  
\item The definition of K-stability is independent of scaling $L\to L^s$ for any $s$, so it makes sense for $L$ a $\Q$-line bundle. In practice, given a test configuration, one can scale the polarisation $\scL\to\scL^r$ for any $r\in\Q$ without affecting the positivity of the Donaldson-Futaki invariant. Remark that this changes the \emph{exponent} of the test configuration.
\end{itemize}\end{remark}

The notion of K-stability extends readily to pairs. First note that, given a divisor $D \subset X$, and a test configuration $(\scX,\scL)$ for $(X,L)$, one obtains a test configuration (however, with possibly non-normal total space) for $(D,L|_D)$ by taking the total space to be the closure $\scD$ of $D$ under the $\C^*$-action. Although $\scD$ may not be normal, this will not affect our definition. Essentially here one needs to check flatness, which is automatic by \cite[Tag 083P]{stacks}. We therefore have corresponding Hilbert and weight polynomials for the test configuration $(\scD,\scL|_{\scD})$, which we denote respectively by \begin{align*}&\tilde{h}(k)=\tilde{a}_0k^{n-1}+O(k^{n-2}), \\ &\tilde{w}(k) = \tilde{b}_0k^{n}+O(k^{n-1}).\end{align*}

\begin{definition}We define the \emph{log Donaldson-Futaki invariant} of $((\scX,\scD); \scL)$ with angle $\beta$ for $0\leq\beta \leq 1$ to be \[ \DF((\scX,\scD);\scL)= \frac{b_0a_1 - b_1 a_0}{a_0} + (1-\beta)\frac{a_0\tilde{b}_0-b_0\tilde{a}_0}{2a_0}.\] We say that $((X,D);L)$ is \emph{log K-stable with angle} $\beta$ if $\DF_{\beta}((\scX,\scD);\scL)>0$ for all test configurations $((\scX,\scD);\scL)$ with $\scX$ normal and non-trivial. Remark that for $\beta=1$ this recovers the definition of K-stability.\end{definition}

The following obvious monotonicity property is very useful.

\begin{lemma}\label{monotonicity} If $(X,L)$ is K-semistable and $((X,D);L)$ is log K-stable with angle $\beta$ then $((X,D);L)$ is log K-stable with angle $\gamma$ for all $\beta \leq \gamma < 1$. \end{lemma}

The definition of the Donaldson-Futaki invariant is quite cumbersome to use in practice. Instead, we use the following intersection theoretic formula due to Li-Xu \cite[Proposition 6]{LX} and Boucksom-Hisamoto-Jonsson \cite{BHJ} in the setting of pairs. There are similar formulas due to Wang \cite[Proposition 17]{XW} and Odaka \cite[Theorem 3.2]{O2}.

\begin{theorem}\label{lxformula}  Denote by \[\mu(X,L) = \frac{-K_X.L^{n-1}}{L^n}\] the slope of $(X,L)$. For a normal test configuration $(\scX, \scL)$ of exponent $r$, the Donaldson-Futaki invariant is given on its compactification over $\pr^1$ as \[\DF(\scX,\scL) = \frac{n}{n+1}\mu(X,L^r)\scL^{n+1} + \scL^n.K_{\scX / \pr^1},  \] with $K_{\scX / \pr^1} = K_\scX - \pi^*K_{\pr^1}$ the relative canonical divisor. Here since the test configuration is normal, its canonical class exists as a Weil divisor and the intersection number makes sense.

For pairs, denote \[\mu_{\beta}((X,L);D) = \frac{-(K_X+(1-\beta)D).L^{n-1}}{L^n}.\] Then similarly \[ \DF_{\beta}((\scX,\scD);\scL) = \frac{n}{n+1}\mu_{\beta}((X,L^r);D) \scL^{n+1} + \scL^n.K_{(\scX, (1-\beta)\scD) / \pr^1},\] with $K_{(\scX, (1-\beta)\scD) / \pr^1} = K_\scX+(1-\beta)\scD - \pi^*K_{\pr^1}.$

\end{theorem}

We will need some simplifications of the K-stability condition. The first, due to Datar-Sz\'ekelyhidi, states that it is enough to consider \emph{equivariant} test configurations.

\begin{definition} Let $G\subset \Aut(X,L)$ be a finite group. We say a test configuration $(\scX,\scL)$ is $G$\emph{-invariant} if $\scX$ admits a global $G$-action lifting to $\scL$ which restricts to the usual action over the non-zero fibres of the test configuration.  \end{definition}

\begin{remark} When $G$ is cyclic, the action on $\scX$ automatically linearises to $\scL$. For general finite groups, the obstruction is the \emph{Schur multiplier} of $G$, which is trivial when $G$ is cyclic \cite[Remark 7.2]{ID}. Note that $G$ automatically commutes with the $\C^*$-action on $(\scX,\scL)$. This is clear away from the central fibre, while on $(\scX_0,\scL_0)$ this can be seen by diagonalising the $\C^*$-action on $H^0(\scX_0,\scL_0^k)$.\end{remark}

\begin{theorem}\cite{DS}\label{equivariance} A smooth anti-canonically polarised Fano variety $X$ is K-stable if and only if it is K-stable with respect to $G$-invariant test configurations. \end{theorem}

\begin{remark} The proof of Theorem \ref{equivariance} is entirely analytic, and in fact shows that K-stability with respect to $G$-invariant test configurations is equivalent to the existence of a K\"ahler-Einstein metric; this in turn is equivalent to K-stability. We expect that Theorem \ref{equivariance} holds for arbitrary possibly singular polarised varieties. \end{remark}

The next simplification due to Li-Xu we will need is that one can assume the central fibre of a test configuration is a normal projective variety, and this is even true equivariantly \cite[Remark 42]{CL}.

\begin{theorem}\label{LX}\cite{LX} An anti-canonically polarised Fano variety is K-stable if and only if it is K-stable with respect to test configurations with normal central fibre.  \end{theorem}

\begin{remark}The equivariant version of Theorem \ref{LX} also follows from the work of Datar-Sz\'ekelyhidi \cite{DS}, while the non-equivariant version is also a consequence of the proof of the Yau-Tian-Donaldson conjecture \cite{CDS,GT}. \end{remark}

\section{Proof of main theorem}

Before proving Theorem \ref{intromaintheorem}, we prove two lemmas. First, starting with a certain kind of test configuration for $(Y,-K_Y)$, we produce a test configuration for $((X,D);-K_X)$. Let $G\subset\Aut(Y)$ be the Galois group. 

\begin{lemma}\label{quotient-test-config} Let $(\scY,\scL)$ be a $G$-invariant normal test configuration with normal central fibre for $(Y,-K_Y)$. Then the quotient of $(\scY,\scL)$ by $G$ is a test configuration for $(X,-K_X)$ with normal central fibre.\end{lemma}

\begin{proof}

Denote by $\psi: \scY \to \scX$ the quotient map and let $\scH$ be the line bundle on $\scX$ such that $\psi^*\scH=\scL$. We aim to prove claim $(\scX,\scH)$ is a test configuration for $(X,-K_X)$. First of all, as the $\C^*$-action commutes with the $G$-action on $\scY$, there is an induced $\C^*$-action on $\scX$ lifting to $\scH$. 

Let $\pi: \scY\to\pr^1$ be the projection. Remark that $\pi$ is $G$-invariant, since $G$ fixes the fibres of $\pi$. By the universal property of the quotient, we have an induced morphism $\eta: \scX\to\pr^1$ such that the following diagram commutes.

\begin{center}
\begin{tikzcd}
\scY \arrow{r}{\psi}  \arrow{rd}{\pi} 
  & \scX \arrow{d}{\eta} \\
    & \pr^1
\end{tikzcd}
\end{center}

Since $\scY$ is normal, $\scX$ is also normal \cite[Chapter XII, Lemma 2.2]{curves}. $\eta$ is therefore a dominant morphism from a normal variety to a smooth curve, and is hence flat \cite[Theorem 9.11]{RH}.

The above commutative diagram also implies that the quotient is fibre-wise, i.e. the quotient of the fibre is the fibre of the quotient. We show this over the point $0\in\pr^1$, the general argument being the same. Indeed, the quotient by the finite group $G$ is defined locally on an affine patch by taking the ring of $G$-invariants. For a $G$-invariant affine patch $Spec(A)\subset\scY$ over $0$ we have the following diagram.

\begin{center}
\begin{tikzcd}
  \C[t] \arrow{d}{\alpha} \arrow{r} 
	& \C[t]/(t)  \\
  A  
\end{tikzcd}
\end{center}

The fibre over $0\in \pr^1$ then corresponds to $\Spec(A\otimes_{\C[t]}\C[t]/(t))$, where we consider $A$ as a $\C[t]$-module through the map $\alpha$. The geometric statement that $G$ fixes the fibres means that for each $f\in\C[t]$, the image $\alpha(f)$ is $G$-invariant. This implies that $A$, considered as a $\C[t]$-module, admits a natural structure as a $G$-module. Then the algebra statement corresponding to the fact that the quotient is fibre-wise is that \[A^G\otimes_{\C[t]}\C[t]/(t)\cong (A\otimes_{\C[t]}\C[t]/(t))^G,\] which is obvious since the $G$-action is trivial on $\C[t]/(t)$. 

As the quotient is fibre-wise, the central fibre $\scX_0$ is the quotient of the normal variety $\scY_0$, and is hence normal. For the same reason, the line bundle $\scH$ is relatively ample, since on each fibre it pulls back to an ample line bundle.

Since the Galois group is cyclic, by Riemann-Hurwitz \cite[Lemma 17.1]{BHPV} we have \[K_Y = \nu^*\left(K_X+\left(1-\frac{1}{m}\right)D\right)\sim \nu^*(\xi K_X)\] for some $\xi > 0$, as by assumption $Y$ is Fano. That $\psi^*\scH=\scL$ implies $\scH$ restricts to $-r\xi K_X$ over the nonzero fibres of $\eta$, where $r$ is the exponent of $(\scY,\scL)$. The exponent $r\xi$ may only be rational, and so we interpret this is a test configuration in the sense of Remark \ref{remarks}. This completes the proof that $(\scX,\scH)$ is a test configuration for $(X,-K_X)$.

\end{proof}

For notational convenience we set $L=-rK_Y$ and $H=-r\xi K_X$ so that $L=\nu^*H$ as in the proof above. With this notation, $(\scX,\scH)$ and $(\scY,\scL)$  are test configurations of exponent one for $(X,H)$ and $(Y,L)$ respectively. Let $\scD$ be the closure of the orbit of $D$ in $\scX$ under the $\C^*$-action. We now relate the Donaldson-Futaki invariants of the two test configurations.

\begin{lemma}\label{DFcomparison} The Donaldson-Futaki invariants of $(\scY,\scL)$ and  $((X,D);-K_X)$ are related by $$\DF(\scY,\scL) = m\DF_{\frac{1}{m}}((\scX,\scD);\scH).$$\end{lemma}

\begin{proof}

We use Theorem \ref{lxformula} to compare the log Donaldson-Futaki invariant of $((\scX,\scD);\scL)$ to the Donaldson-Futaki invariant of $(\scY,\scL)$. We have \[\DF(\scY,\scL) = \frac{n}{n+1}\mu(Y,L)\scL^{n+1} + \scL^n.K_{\scY / \pr^1},  \] while \[ \DF_{\frac{1}{m}}((\scX,\scD);\scH) = \frac{n}{n+1}\mu_{\frac{1}{m}}((X,H);D) \scH^{n+1} + \scH^n.K_{(\scX, (1-\frac{1}{m})\scD) / \pr^1}. \] We interpret the Donaldson-Futaki invariants as the degrees of the corresponding $0$-cycles given by the intersection-theoretic formulae, and show that \[\psi_*\DF(\scY,\scL) = m\DF_{\frac{1}{m}}((\scX,\scD);\scH).\] Pushforward under proper maps preserves the degree of $0$-cycles, so this will be enough to conclude that $\DF(\scY,\scL)>0$ and hence $(Y,-K_Y)$ is K-stable.

We first show $\mu_{\frac{1}{m}}((X,H);D) = \mu(Y,L)$. Since $\nu^*H=L$ and $K_Y = \nu^*(K_X+(1-\frac{1}{m})D)$, we have \begin{align*} \mu(Y,L)&=\frac{-K_Y.L^{n-1}}{L^n}, \\ &= \frac{\nu^*(K_X+(1-\frac{1}{m})D).(\nu^*H)^{n-1}}{(\nu^*H)^n}\\ &=\frac{\nu^*(-(K_X+(1-\frac{1}{m})D).H^{n-1})}{\nu^*(H^n)}\\ &= \mu_{\frac{1}{m}}((X,H);D),\end{align*} where we have used that $\nu$ is flat.

A similar idea works for comparing the Donaldson-Futaki invariants. Note that one can define the pullback of a Weil divisor $\scW$ under the finite map $\psi$ by restricting to the smooth locus of $\scX$, pulling back and taking the closure to give a Weil divisor $\psi^*\scW$ on $\scY$ \cite[Section 2.40]{JK}. By Riemann-Hurwitz for normal varieties \cite[Section 2.3]{JK}, there is a \emph{canonically defined} $\Q$-Weil divisor $R$ contained in the locus where $\psi$ is not \'etale such that \[K_{\scY} = \psi^*K_{\scX} + R.\] We claim $R = (1-\frac{1}{m})\psi^*\scD.$ This is certainly true away from the central fibre, so we need to ensure that $R$ does not contain an irreducible component of $\scY_0$.

We first we show that the $G$-action on $\scY_0$ is non-trivial. Indeed, otherwise we would have an isomorphism $(\scY_0,\scL_0) \cong (\scX_0,\scH_0)$; this would in particular give equality of the corresponding Hilbert polynomials. But the Hilbert polynomials of the central fibres of the test configurations are equal to the Hilbert polynomials of the respective general fibres, so this would imply equality of the Hilbert polynomials of $(Y,L)$ and $(X,H)$. This is clearly false, as for example by asymptotic Riemann-Roch the lead term of the Hilbert polynomial of $(Y,L)$ is given as \[\frac{L^n}{n!} = \frac{\nu^*(H^n)}{n!} \neq \frac{H^n}{n!}.\] 

Recall that $\scY_0$ is normal. Because we have a non-trivial $G$-action on $\scY_0$, there is a codimension one subscheme $P$ of $\scY_0$ away from which the map $\scY_0 \to \scX_0$ is \'etale. This implies the $G$-action on $\scY$ is \'etale away from $\scD\cup P$. This in turn implies $R=(1-\frac{1}{m})\psi^*\scD$, as $R$ is of pure codimension one. Note also that $\pi^*K_{\pr^1}=\psi^*\eta^* K_{\pr^1}$, since the relevant diagram commutes. So \begin{equation}\label{weildivisors} K_{\scY} - \pi^*K_{\pr^1} = \psi^*\left(K_{\scX} + \left(1-\frac{1}{m}\scD\right) - \eta^*K_{\pr^1}\right).\end{equation}

Remark that the definition of the pullback of Weil divisors under $\psi$ given above agrees with Fulton's definition of the pullback of \emph{cycles} under finite quotient maps \cite[Example 1.7.6]{fulton}, where one takes the cycle associated with the preimage of the divisor with appropriate multiplicity. Indeed, this is certainly true on the smooth locus, and since $\scY$ is normal, its singular locus is of codimension two. Hence equation (\ref{weildivisors}) holds at the level of cycles. Again by \cite[Example 1.7.6]{fulton}, for $\alpha$ a cycle on $\scX$ we have $\psi_*\psi^*\alpha = m\alpha$. The projection formula (e.g. \cite[p17]{positivity}) in this setting states that, for $D_1,\hdots,D_n$ arbitrary Cartier divisors on $\scX$, and $\delta$ a cycle on $\scY$, we have \[\psi_*(\psi^*D_1\hdots\psi^*D_n.\delta) = D_1\hdots D_n.\psi_*\delta.\]

It follows that 
\begin{align*}\psi_*(\DF(\scY,\scL)) &= \psi_*\left(\frac{n}{n+1}\mu(Y,L)\scL^{n+1} + \scL^n.K_{\scY / \pr^1}\right),\\ 
&= \psi_*\left(\frac{n}{n+1}\mu_{\frac{1}{m}}((X,H);D) (\psi^*\scH)^{n+1} + (\psi^*\scH)^n.\psi^*K_{({\scX},(1-\frac{1}{m})\scD)/ \pr^1}\right),\\
 &= \left(\frac{n}{n+1}\mu_{\frac{1}{m}}((X,H);D) \scH^{n}.(\psi_*\psi^*\scH) + \scH^n.\psi_*\psi^*K_{({\scX},(1-\frac{1}{m})\scD)/ \pr^1}\right), \\
  &= m\left(\frac{n}{n+1}\mu_{\frac{1}{m}}((X,H);D) \scH^{n+1} + \scH^n.K_{({\scX},(1-\frac{1}{m})\scD)/ \pr^1}\right), \\
 &= m\DF_{\frac{1}{m}}((\scX,\scD);\scH). \end{align*} This concludes the proof.

\end{proof}

The following is now a straightforward corollary.

\begin{corollary}\label{oneway}Let $\nu: Y\to X$ be a cyclic Galois cover of Fano varieties of degree $m$ with branch divisor $D\subset X$, and with $Y$ smooth. Then $(Y,-K_Y)$ is K-stable provided $((X,D);-K_X)$ is log K-stable with angle $\frac{1}{m}$.\end{corollary}

\begin{proof}

To show $(Y,-K_Y)$ is K-stable, it suffices to show the Donaldson-Futaki invariant of each $G$-invariant test configuration with normal central fibre is strictly positive. This will be enough to conclude K-stability, with the $G$-invariance being justified by Theorem \ref{equivariance} as $Y$ is smooth, and the normality of the central fibre being justified by Theorem \ref{LX}. This then follows immediately from Lemmas \ref{quotient-test-config} and \ref{DFcomparison}.

\end{proof}

\begin{remark} Most of the above argument works for general polarised varieties. What we need is essentially that the polarisation on $X$ pulls back to the polarisation on $Y$, which is true by Riemann-Hurwitz in our case. Then the conclusion of Corollary \ref{oneway} holds for all $G$-invariant test configurations with normal central fibre. \end{remark}

We now use a monotonicity argument to conclude the proof of Theorem \ref{intromaintheorem}.

\begin{theorem}\label{bodymaintheorem} Let $Y\to X$ be a cyclic Galois covering of smooth Fano varieties with smooth branch divisor $D\in |-\lambda K_X|$ for $\lambda \geq 1$. If $X$ is K-semistable, then $Y$ is K-stable. \end{theorem}

\begin{proof} If $\lambda=1$, a result of Odaka-Sun \cite[Theorem 5.4]{OS2} implies that $(X,D)$ is log K-stable for all sufficiently small $\beta>0$. Suppose now $\lambda>1$, so the pair $(X,D)$ is log general type, i.e. $K_X+D$ is ample. Then $K_X+(1-\beta)D$ is also ample for all $\beta>0$ sufficiently small, and by another result of Odaka-Sun \cite[Theorem 4.1]{OS2} $(X,D)$ is again log K-stable for all $\beta>0$ sufficiently small. 

By assumption $X$ is K-semistable. Lemma \ref{monotonicity} then implies $(X,D)$ is log K-stable for all $0 < \beta <1$. The hypotheses of Corollary \ref{oneway} therefore hold for cyclic Galois covers of \emph{all} degrees; it follows that $Y$ is K-stable. \end{proof}

\begin{remark} Even for projective space, it is a subtle problem to determine for which hypersurfaces $D\subset\pr^n$ and for which angles $\beta$ the pair $(\pr^n,D)$ is log K-stable with angle $\beta$ (even though it is automatic for $\lambda \geq 1$ as above). For example, if $D$ is of degree $d$ and $\beta < (\frac{n+1}{d}-1)n$ (e.g. if $D$ is a line this holds for all $\beta<1$), then $(\pr^n,D)$ is \emph{not} log K-stable \cite[Proposition 12]{LS}. \end{remark}

\begin{corollary}\label{automcor} Under the hypotheses of Theorem \ref{bodymaintheorem}, $Y$ has finite automorphism group. \end{corollary}

\begin{proof} By Theorem \ref{bodymaintheorem}, $Y$ is K-stable and hence admits a K\"ahler-Einstein metric. It then follows from Matsushima's Theorem \cite{YM} that $\Aut(Y,-K_Y)$ is reductive. 

Suppose there is a subgroup of $\Aut(Y,-K_Y)$ isomorphic to $\C^*$. This generates a test configuration $(\scY,\scL)=(Y\times\C,-K_Y)$, with the action on the central fibre given by the $\C^*$-subgroup. The inverse of this $\C^*$-subgroup generates another test configuration, say $(\scY',\scL')$, such that the total weight of the $\C^*$-action on $H^0(\scY_0,\scL_0^k)$ equals \emph{minus} the total weight of the $\C^*$-action on $H^0(\scY_0',\scL_0'^k)$. It follows that \[\DF(\scY,\scL) = -\DF(\scY',\scL').\] On the other hand, since $Y$ is K-stable, the Donaldson-Futaki invariants of \emph{all} test configurations are strictly positive. Hence $\Aut(Y,-K_Y)$ admits no $\C^*$-subgroups. The result follows since reductive groups with no $\C^*$-subgroups are finite. \end{proof}

\begin{remark} The same technique was used by Odaka to show that smooth varieties with ample canonical class have finite automorphism group \cite{O1}. Corollary \ref{automcor} relies on the proof of the Yau-Tian-Donaldson conjecture, namely that K-stability implies the existence of a K\"ahler-Einstein metric. There is no direct proof that K-stable varieties have reductive automorphism group, though some progress has been made \cite{CD}. \end{remark}

\section{Examples}\label{examples}

We give some rather elementary applications of Theorem \ref{intromaintheorem}, mostly by comparison with the work of Arezzo-Ghigi-Pirola \cite{AGP}. In most examples we take the base $X$ to either be projective space or a product of projective spaces. Since these admit K\"ahler-Einstein metrics, they are K-poystable \cite{RB}. It is an open problem to give an algebro-geometric proof of K-polystability of projective space.

\begin{example}[Hypersurfaces]\label{hypersurface} As mentioned in the introduction, an arbitrary smooth hypersurface of the form \[Y=V(x_0^{n}+f(x_1,\hdots, x_n))\subset \pr^{n}\] satisfies the criterion of Theorem \ref{intromaintheorem}. Arezzo-Ghigi-Pirola \cite[Proposition 3.1]{AGP} show that \[Y=V(x_0^{n}+x_1^n+f(x_2,\hdots, x_n))\subset \pr^{n}\] admits a K\"ahler-Einstein metric, while Pukhlikov's work implies that provided $Y$ is \emph{general} and $n\geq 6$, it admits a K\"ahler-Einstein metric \cite[Theorem 2]{AP}. \end{example}

\begin{example}[Double solids]\label{doublesolid} Arbitrary double covers are Galois; this corresponds to the fact that all field extensions of degree two are Galois. Let $Y$ be a Fano manifold which can be realised as a double cover of $\pr^n$ branched over a smooth hypersurface of degree $d$. Such Fano varieties are called double solids. By Riemann-Hurwitz, for $Y$ to be Fano requires $d\leq 2n$, while we require $d\geq n+1$ for the hypotheses of Theorem \ref{intromaintheorem} to apply. Thus provided $D$ is smooth, $Y$ admits a K\"ahler-Einstein metric for \[n+1 \leq d \leq 2n.\] Arezzo-Ghigi-Pirola \cite[Theorem 3.2]{AGP} showed under the same hypotheses that $Y$ admits a K\"ahler-Einstein metric for \[n+2\leq d\leq 2n,\] so Theorem \ref{intromaintheorem} again gives marginally better results. In particular, for $d=n+1$, Corollary \ref{KEmetrics} gives new examples of Fano manifolds admitting K\"ahler-Einstein metrics. Remark that, by Riemann-Hurwitz, this requires $n$ to be odd. \end{example}

\begin{example}[del Pezzos] In the case $n=3$, the hypersurfaces described in Example \ref{hypersurface} are just cubic surfaces, isomorphic to $\pr^2$ blown up at $6$ points. The surface case of the double solids with $d=4$ are isomorphic to $\pr^2$ blown up at $7$ points. These examples were given by Li-Sun \cite[Section 6]{LS}. One can similarly describe $\pr^2$ blown up at certain configurations of $8$ points as a double cover over $\pr(1:1:2)$, however as this is not smooth and it does not admit an orbifold K\"ahler-Einstein metric, Theorem \ref{intromaintheorem} does not apply.  \end{example}

\begin{example}[Threefolds] In the classification of Fano threefolds, several are defined as coverings of other Fano threefolds, and so one finds many applications of Theorem \ref{intromaintheorem}. 

\begin{itemize} \item Let $Y\to\pr^3$ be a double cover with branch divisor a sextic surface. This is another example of a double solid. In more precise terminology, $Y$ is a hyperelliptic threefold, as it has Picard rank one and its anti-canonical system determines a double cover to another Fano threefold. Theorem \ref{intromaintheorem} then implies that $Y$ is K-stable, and hence that $Y$ admits a K\"ahler-Einstein metric. This was first proven by Arezzo-Ghigi-Pirola \cite[Section 3]{AGP}. We remark that Arezzo-Ghigi-Pirola also give other examples of hyperelliptic threefolds with K\"ahler-Einstein metrics to which Theorem \ref{intromaintheorem} does not obviously apply. 
\item Let $Y\to \pr^1\times\pr^1\times\pr^1$ be a double cover branched over a divisor of tridegree $(2,2,2)$. It follows from Theorem \ref{intromaintheorem} that $Y$ is K-stable.
\item Let $X$ be a hypersurface of bidegree $(1,1)$ in $\pr^2\times\pr^2$. Alternatively $X$ can be described as a flag variety of $SL_3$ \cite[Example 3.2.3.2]{cox}. As $X$ is a smooth flag variety, it admits a K\"ahler-Einstein metric, and hence is K-polystable. Let $Y\to X$ be a double cover branched over a smooth element of $|-K_X|$. Again it follows that $X$ is K-stable. 
\end{itemize} 

It was asked by Codogni-Fanelli-Svaldi-Tasin if every Fano variety which appears as a fibre of a Mori fibre space is K-semistable \cite{CFST}. They show that each of the above Fano threefolds appears as the fibre of a Mori fibre space (note that the first example is trivial as it has Picard rank one, while the last two examples both have Picard rank two). This answers some special cases of their question. 

 \end{example}

\begin{example}[Alpha invariants] Given a pair $(X,D)$ with $X$ Fano and $D\in |-K_X|$, one can obtain lower bounds on the angles $\beta$ such that $(X,D)$ is log K-stable using Tian's alpha invariant \cite[Theorem 5.4]{OS2}. One could use this to generate more examples of covers $Y\to X$ with branch divisor $D$ such that $Y$ is K-stable using Theorem \ref{introequivalence}, even if $X$ itself is not K-stable.\end{example}

\printbibliography

\vspace{4mm}

\noindent Ruadha\'i Dervan, University of Cambridge,  UK and Universit\'e libre de Bruxelles, Belgium. \\ R.Dervan@dpmms.cam.ac.uk

\end{document}